\documentclass[11pt, reqno]{amsart}
\usepackage{amsfonts}
\usepackage{amsmath}
\usepackage{amssymb, color}
\usepackage{amscd}
\usepackage[mathscr]{eucal}
\usepackage{url}

\allowdisplaybreaks[1]

\renewcommand*{\bar}{\overline}
\newcommand{\gfp}[1]{\Gamma_p{\left({#1}\right)}}
\newcommand{\biggfp}[1]{\Gamma_p{\bigl({#1}\bigr)}}

\theoremstyle{plain}

\newtheorem{theorem}{Theorem}[section]
\newtheorem{lemma}[theorem]{Lemma}

\newtheorem{cor}[theorem]{Corollary}
\theoremstyle{definition}
\newtheorem{defi}[theorem]{Definition}

\numberwithin{equation}{section}

\makeatletter
\def\imod#1{\allowbreak\mkern5mu({\operator@font mod}\,\,#1)}
\makeatother

\begin{document}

\title[Dwork hypersurfaces and hypergeometric functions]{The number of $\mathbb{F}_p$-points on Dwork hypersurfaces\\ and hypergeometric functions}

%\author{Jenny G. Fuselier}
%\address{Jenny G. Fuselier, Department of Mathematics \& Computer Science\\
%Drawer 31, High Point University\\
%High Point, NC 27268 \\
%USA}
%\email{jfuselie@highpoint.edu}

\author{Dermot M\lowercase{c}Carthy}
\address{Dermot M\lowercase{c}Carthy, Department of Mathematics \& Statistics\\
Texas Tech University\\
Lubbock, TX 79410-1042\\
USA}
\email{dermot.mccarthy@ttu.edu}

\thanks{This work was supported by a grant from the Simons Foundation (\#353329, Dermot McCarthy).}

%\url{http://maths.ucd.ie/$\sim$mccarthy}

%\subjclass[2010]{Primary: 11F33, 33C20; Secondary: 11S80, 33E50}

%\date{July 20, 2009}

\begin{abstract}
We provide a formula for the number of $\mathbb{F}_{p}$-points on the Dwork hypersurface
$$x_1^n  + x_2^n \dots + x_n^n - n \lambda \, x_1 x_2 \dots x_n=0$$ 
in terms of a $p$-adic hypergeometric function previously defined by the author. This formula holds in the general case, i.e for any $n, \lambda \in \mathbb{F}_p^{*}$ and for all odd primes $p$, thus extending results of Goodson and Barman et al which hold in certain special cases.
\end{abstract}

\maketitle

%%%%%%%%%%%%%%%%%%%%%%%%%%%%%%%%%%%%%%%%%%%%%%%%%%
%%%%%%%%%%%%%%%%%%%%%%%%%%%%%%%%%%%%%%%%%%%%%%%%%%
%%%%%%%%%%%%%%%%%%%%%%%%%%%%%%%%%%%%%%%%%%%%%%%%%%
%%%%%%%%%%%%%%%%%%%%%%%%%%%%%%%%%%%%%%%%%%%%%%%%%%
%%%%%%%%%%%%%%%%%%%%%%%%%%%%%%%%%%%%%%%%%%%%%%%%%%
%%%%%%%%%%%%%%%%%%%%%%%%%%%%%%%%%%%%%%%%%%%%%%%%%%

\section{Introduction}\label{sec_Intro}

The first part of the Weil conjectures, the rationality of the zeta-function of algebraic varieties over finite fields, was proved by Dwork in a 1960 paper \cite{Dw1} using $p$-adic analysis. Subsequently, Dwork further developed his $p$-adic techniques and studied the special case of zeta functions of non-singular projective hypersurfaces. In particular, he examined how his $p$-adic constructions varied within a family, his so-called \emph{deformation theory} \cite{Dw2, Dw3}. The family
\begin{equation}\label{Dwork_HypS}
 x_1^n  + x_2^n \dots + x_n^n - n \lambda \, x_1 x_2 \dots x_n=0
 \end{equation}
appears often in his work and is now known as the Dwork family of hypersurfaces.
In the early 1990's, the $n=5$ case appeared in the celebrated work of Candelas et al \cite{COGP} on mirror symmetry, thus reviving interest in the Dwork family.

More recently, formulas for counting the number of points over finite fields on the Dwork hypersurface using hypergeometric functions have been of special interest. Let $\mathbb{F}_{q}$ denote the finite field with $q$ elements, where $q$ is a power of a prime $p$. Koblitz \cite{Ko4} provides a formula for the number of $\mathbb{F}_{q}$-points on monomial deformations of a diagonal hypersurface, of which the Dwork family is a special case, in terms of Gauss sums, when $q$ is in a particular congruence class. He then highlights the analogy between this formula and the Barnes integral for classical hypergeometric series. In \cite{McC5}, this author provided a simple formula for the number of $\mathbb{F}_{p}$-points on the Dwork hypersurface, in the case ${n=5}$ and ${\lambda=1}$, in terms of a finite field hypergeometric function, when $p \equiv 1 \pmod 5$. 
%(See section X below for more information of finite field hypergeometric series.) 
We then extended this result to all odd primes using a hypergeometric type function defined in terms of the $p$-adic gamma function (see Definition \ref{def_Gp} below).
In \cite{Go}, Goodson considers the ${n=4}$ case and gives formulas for the number of $\mathbb{F}_{q}$-points in terms of finite field hypergeometric functions when $q \equiv 1 \pmod 4$, and extends to all odd primes using this author's $p$-adic hypergeometric function. She also conjectures a formula in the special case that $n$ is prime and that $p \not\equiv 1 \pmod n$, which was proven by Barman et al in \cite{BRS}. 
  
The purpose of this paper is to provide a formula for the number of $\mathbb{F}_{p}$-points on the Dwork hypersurface in terms of the $p$-adic hypergeometric function in the general case, i.e for any $n$ and $\lambda$ in $\mathbb{F}_p^{*}$, and which holds for all odd primes $p$.

%%%%%%%%%%%%%%%%%%%%%%%%%%%%%%%%%%%%%%%%%%%%%%%%%%
%%%%%%%%%%%%%%%%%%%%%%%%%%%%%%%%%%%%%%%%%%%%%%%%%%
%%%%%%%%%%%%%%%%%%%%%%%%%%%%%%%%%%%%%%%%%%%%%%%%%%
%%%%%%%%%%%%%%%%%%%%%%%%%%%%%%%%%%%%%%%%%%%%%%%%%%
%%%%%%%%%%%%%%%%%%%%%%%%%%%%%%%%%%%%%%%%%%%%%%%%%%
%%%%%%%%%%%%%%%%%%%%%%%%%%%%%%%%%%%%%%%%%%%%%%%%%%

 \section{Statement of Results}\label{sec_Results}

We first define the $p$-adic hypergeometric function. Let $\gfp{\cdot}$ denote Morita's $p$-adic gamma function and let $\omega$ denote the Teichm\"{u}ller character of $\mathbb{F}_p$ with $\bar{\omega}$ denoting its character inverse. 
For $x \in \mathbb{Q}$ we let  $\left\lfloor x \right\rfloor$ denote the greatest integer less than or equal to $x$ and
$\langle x \rangle$ denote the fractional part of $x$, i.e. $x- \left\lfloor x \right\rfloor$.
\begin{defi}\cite[Definition 1.1]{McC7}\label{def_Gp}
Let $p$ be an odd prime and let $x \in \mathbb{F}_p$. For $m \in \mathbb{Z}^{+}$ and $1 \leq i \leq m$, let $a_i, b_i \in \mathbb{Q} \cap \mathbb{Z}_p$.
Then we define  
\begin{multline*}%\label{for_GFn}
{_{m}G_{m}}
\biggl[ \begin{array}{cccc} a_1, & a_2, & \dotsc, & a_m \\
 b_1, & b_2, & \dotsc, & b_m \end{array}
\Big| \; x \; \biggr]_p
: = \frac{-1}{p-1}  \sum_{j=0}^{p-2} 
(-1)^{jm}\;
\bar{\omega}^j(x)\\
\times \prod_{i=1}^{m} 
\frac{\biggfp{\langle a_i -\frac{j}{p-1}\rangle}}{\biggfp{\langle a_i \rangle}}
\frac{\biggfp{\langle - b_i +\frac{j}{p-1}\rangle}}{\biggfp{\langle -b_i \rangle}}
%(-p)^{-\lfloor{a_i-\frac{j}{p-1}}\rfloor -\lfloor{-b_i+\frac{j}{p-1}}\rfloor + \lfloor{a_i}\rfloor + \lfloor{-b_i}\rfloor}.
(-p)^{-\lfloor{\langle a_i \rangle -\frac{j}{p-1}}\rfloor -\lfloor{\langle -b_i \rangle +\frac{j}{p-1}}\rfloor}.
\end{multline*}
\end{defi}
\noindent Throughout the paper we will refer to this function as ${_{m}G_{m}}[\cdots]$.
We note that the value of ${_{m}G_{m}}[\cdots]$ depends only on the fractional part of the $a$ and $b$ parameters, and is invariant if we change the order of the parameters.

We now describe our main result. We consider the Dwork hypersurface as described in (\ref{Dwork_HypS}).
Let $d:=\gcd({p-1},n)$
%$t:=\frac{p-1}{d}$ 
and
\begin{equation}\label{Def_W}
W:=\{w=(w_1, w_2, \ldots, w_n) \in \mathbb{Z}^n : 0 \leq w_i < d, \sum_{i=1}^n w_i \equiv 0 \pmod d\}. 
\end{equation} 
Define an equivalence relation $\sim$ on $W$ by 
\begin{equation}\label{Def_tilde}
w \sim w \prime \textup{ if } w- w\prime \textup{ is a multiple modulo $d$ of $(1,1, \ldots, 1)$}. 
\end{equation}
We note $|W|=d^{n-1}$ and $|{W/\sim}|=d^{n-1}$ as every equivalence class has $d$ elements. We will denote the class containing $w$ by $[w]$. We note also that each class contains a representative $w$ where some $w_i = 0$, for $1 \leq i \leq n$. We will write $[w^{\ast}]$ to indicate that we have chosen such a representative for a particular class.

For a given $w=(w_1, w_2, \ldots, w_n) \in W$, define $n_k$ to be the number of $k$'s appearing in $w$, i.e.,
$n_k = |\{w_i \mid  1 \leq i \leq n, w_i=k \} |.$
We then let $S_w := \{ k \mid 0 \leq k \leq {d-1}, n_k=0 \}$ and $S_w^{c}$ denote its complement in $\{0,1 \cdots, {d-1}\}$. So the elements of $S_w$ are the numbers from $0$ to ${d-1}$, inclusive, which do not appear in $w$.
We define the following lists
\begin{equation}\label{def_Aw}
A_w: \left[ \tfrac{d-k}{d} \mid k \in S_w \right] \cup \left[ \tfrac{h}{n} \mid 0 \leq h \leq n-1, h \not\equiv 0 \imod{\tfrac{n}{d}} \right];
\end{equation}
\begin{equation}\label{def_Bw}
B_w: \left[  \tfrac{d-k}{d} \, \textup{repeated $n_k$-1 times} \mid k \in S_w^c \right] .
\end{equation}
We note both lists contain 
\begin{equation}\label{def_s}
s:=n-|S_w^c|
\end{equation}
numbers.

\begin{theorem}\label{thm_Main}
For a prime $p$, let $N_p(\lambda)$ be the number of points in $\mathbb{P}^{n-1}(\mathbb{F}_p)$ on 
$$ x_1^n  + x_2^n \dots + x_n^n - n \lambda \, x_1 x_2 \dots x_n=0,$$
for some $n, \lambda \in \mathbb{F}_p^{*}$. 
Define $d:=\gcd({p-1},n)$ and let $W$, $\sim$, $A_w$, $B_w$ and $s$ be defined by (\ref{Def_W})
%, (\ref{Def_tilde}), (\ref{def_Aw}) and (\ref{def_Bw}) 
- (\ref{def_s})
respectively. 
 Then for $p$ odd,
\begin{equation*}
N_p(\lambda) = \frac{p^{n-1}-1}{p-1} + (-1)^n \sum_{[w^{\ast}] \in W/\sim} (-p)^{\frac{1}{d} \sum_{i=1}^{n} w_i} \;
\prod_{i=1}^{n} \gfp{\tfrac{w_i}{d}} \;
{_{s}G_{s}}
\biggl[ \begin{array}{c} A_w \\
 B_w \end{array}
\Big| \; \lambda^n \; \biggr]_p.
\end{equation*}
\end{theorem}
\noindent If $p \mid n$ then the problem reduces to the $\lambda=0$ case, formulas for which are well known and can be found in \cite{IR, W}. 
We now look at a couple of special cases.

\begin{cor}\label{cor_RelPrime} 
If $d=\gcd({p-1},n) = 1$ then
\begin{equation*}
N_p(\lambda) = \frac{p^{n-1}-1}{p-1} + (-1)^n  \;
{_{n-1}G_{n-1}}
\biggl[ \begin{array}{cccc} \frac{1}{n} & \frac{2}{n} & \dots & \frac{n-1}{n} \\
 1 & 1 & \dots & 1 \end{array}
\Big| \; \lambda^n \; \biggr]_p.
\end{equation*}
\end{cor}
\noindent Corollary \ref{cor_RelPrime} is a slight generalization of the result we mentioned in Section \ref{sec_Intro} for the case $n$ is prime and $p \not\equiv 1 \pmod n$, which was conjectured by Goodson \cite{Go} and proven by Barman et al in \cite{BRS}.

As noted in Section \ref{sec_Intro} some of the previous results counting the number of $\mathbb{F}_{p}$-points on the Dwork hypersurface have been in terms of finite field hypergeometric functions and were valid only for primes in certain congruence classes. This restriction to primes in certain congruence classes is a common theme in results involving finite field hypergeometric functions. Establishing results for all primes is the main reason we developed the $p$-adic function defined above. (See \cite{McC7} for a more complete discussion.) Conversely, if we choose to restrict results involving ${_{m}G_{m}}[\cdots]$ to primes in certain congruence classes then it is always possible to reduce these results to expressions in terms of $_mF_m(\cdots)$. Our next corollary does exactly that for Theorem \ref{thm_Main} in the case $p \equiv 1 \pmod n$.

%Hypergeometric functions over finite fields were originally defined by Greene \cite{G}, who first established these functions as analogues of classical hypergeometric functions.
%Functions of this type were also introduced by Katz \cite{K} about the same time.
%In the present article we use a normalized version of these functions defined by this author \cite{McC6}, which is more suitable for our purposes.
%The reader is directed to \cite[\S 2]{McC6} for the precise connections among these three classes of functions.

%Let $\mathbb{F}_{p}$ denote the finite field with $p$ elements, $p$ a prime, and 
Let $\widehat{\mathbb{F}^{*}_{q}}$ denote the group of multiplicative characters of $\mathbb{F}^{*}_{q}$.
We extend the domain of $\chi \in \widehat{\mathbb{F}^{*}_{q}}$ to $\mathbb{F}_{q}$ by defining $\chi(0):=0$ (including for the trivial character $\varepsilon$) and denote $\bar{\chi}$ as the inverse of $\chi$.
Let $\theta$ be a fixed non-trivial additive character of $\mathbb{F}_q$ and for $\chi \in \widehat{\mathbb{F}^{*}_{q}}$ we define the Gauss sum $g(\chi):= \sum_{x \in \mathbb{F}_p} \chi(x) \theta(x)$.
We define the finite field hypergeometric function as follows. (See \cite{McC6} for the relationship to the functions of Greene and Katz.)
\begin{defi}\cite[Definition 1.4]{McC6}\label{def_F} 
For $A_1, A_2, \dotsc, A_m, B_1, B_2 \dotsc, B_m \in \widehat{\mathbb{F}_q^{*}}$ and $x \in \mathbb{F}_q$,
\begin{equation*}\label{def_HypFnFF}
{_{m}F_{m}} {\biggl( \begin{array}{cccc} A_1, & A_2, & \dotsc, & A_m \\
 B_1, & B_2, & \dotsc, & B_m \end{array}
\Big| \; x \biggr)}_{q}\\
:= \frac{-1}{p-1}  \sum_{\chi \in \widehat{\mathbb{F}_p^{*}}}
\prod_{i=1}^{m} \frac{g(A_i \chi)}{g(A_i)} \frac{g(\bar{B_i \chi})}{g(\bar{B_i})}
%\prod_{i=1}^{n} \frac{g(A_i \chi)}{g(A_i)}
%\prod_{j=1}^{n} \frac{g(\bar{B_j \chi})}{g(\bar{B_j})}
 %g(\bar{\chi})
 \chi(-1)^{m}
 \chi(x).
\end{equation*}
\end{defi}
\noindent We note Definition \ref{def_HypFnFF} is stated slightly differently than the original in \cite{McC6}, which has an implied $B_1 = \varepsilon$, as is often the custom in hypergeometric functions. The relationship between $_mF_m(\dots)$ and ${_{m}G_{m}}[\cdots]$ is outlined in \cite{McC7} and reproduced below in a slightly altered form to take account of the altered definition of $_mF_m(\dots)$ above.
\begin{lemma}[\cite{McC7} Lemma 3.3]\label{lem_G_to_F}
For a fixed odd prime $p$, let $A_i, B_k \in \widehat{\mathbb{F}_p^{*}}$ be given by $\bar{\omega}^{a_i(p-1)}$ and $\bar{\omega}^{b_k(p-1)}$ respectively, where $\omega$ is the Teichm\"{u}ller character . Then
\begin{equation*}
{_{m}F_{m}} {\biggl( \begin{array}{cccc} A_1, & A_2, & \dotsc, & A_m \\
 B_1, & B_2, & \dotsc, & B_m \end{array}
\Big| \; t \biggr)}_{p}
=
{_{m}G_{m}}
\biggl[ \begin{array}{cccc} a_1, & a_2, & \dotsc, & a_m \\
 b_1, & b_2, & \dotsc, & b_m \end{array}
\Big| \; t^{-1} \; \biggr]_p.
\end{equation*}
\end{lemma}
Let $p \equiv 1 \pmod{n}$ and so $d:=\gcd(p-1,n) = n$ and $t:=\frac{p-1}{d}=\frac{p-1}{n}$. Let $T$ be a fixed generator for $\widehat{\mathbb{F}_p^{*}}$ and define the following lists
\begin{equation*}\label{def_ATw}
A_{T,w}^{\prime}: \left[ T^{(n-k)t} \mid k \in S_w \right] ;
\end{equation*}
\begin{equation*}\label{def_BTw}
B_{T,w}^{\prime}: \left[  T^{(n-k)t} \, \textup{repeated $n_k$-1 times} \mid k \in S_w^c \right] .
\end{equation*}

\begin{cor}\label{cor_FF}
If $d=\gcd(p-1,n) = n$, i.e., $p \equiv 1 \pmod{n}$, and $t:=\frac{p-1}{d}$, then
\begin{equation*}
N_p(\lambda) = \frac{p^{n-1}-1}{p-1} +  \sum_{[w^{\ast}]  \in W/\sim}  
\prod_{i=1}^{n} g(T^{w_i t})\;
{_{|S_w|}F_{|S_w|}} {\biggl( \begin{array}{c} A_{T,w}^{\prime} \\
 B_{T,w}^{\prime} \end{array}
\Big| \; \lambda^{-n} \; \biggr)}_{p}.\\
\end{equation*}
\end{cor}

%%%%%%%%%%%%%%%%%%%%%%%%%%%%%%%%%%%%%%%%%%%%%%%%%%
%%%%%%%%%%%%%%%%%%%%%%%%%%%%%%%%%%%%%%%%%%%%%%%%%%
%%%%%%%%%%%%%%%%%%%%%%%%%%%%%%%%%%%%%%%%%%%%%%%%%%
%%%%%%%%%%%%%%%%%%%%%%%%%%%%%%%%%%%%%%%%%%%%%%%%%%
%%%%%%%%%%%%%%%%%%%%%%%%%%%%%%%%%%%%%%%%%%%%%%%%%%
%%%%%%%%%%%%%%%%%%%%%%%%%%%%%%%%%%%%%%%%%%%%%%%%%%

 \section{An Example}\label{sec_Examples}
In this section we give an example of how Theorem \ref{thm_Main} works in practice, when $n=4$.
Let $p$ be an odd prime and let $N_p(\lambda)$ be the number of points in $\mathbb{P}^{3}(\mathbb{F}_p)$ on 
$$ x_1^4  + x_2^4 +x_3^4 + x_4^4 - 4 \lambda \, x_1 x_2 x_3 x_4=0,$$
for some $\lambda \in \mathbb{F}_p^{*}$.
%\underline{Case $d=1$:} If $d=\gcd(p-1,n)=1$ then we can apply Corollary \ref{cor_RelPrime} directly to get
%\begin{equation*}
%N_p(\lambda) = \frac{p^{3}-1}{p-1} + 
%{_{3}G_{3}}
%\biggl[ \begin{array}{ccc} \frac{1}{4} & \frac{1}{2} & \frac{3}{4} \\
% 1 & 1 & 1 \end{array}
%\Big| \; \lambda^4 \; \biggr]_p.
%\end{equation*}
%\underline{Case $p \equiv 3 \pmod{4}$:} 

If $p \equiv 3 \pmod{4}$ then $d=\gcd({p-1},4)= 2$. We now evaluate the sets $W$ and $W/\sim$. We first note that the contribution of any $w=(w_1, w_2, w_3, w_4)$ to the sum in Theorem \ref{thm_Main} is the same as that for any permutation of $w$. We therefore list the elements of these sets up to permutation. We will however indicate, using a superscript, the total number of distinct permutations. So
$$W = \{ (0,0,0,0), (0,0,1,1)^6, (1,1,1,1) \}$$
and
$${W/\sim} = \{[0,0,0,0], [0,0,1,1]^3 \}.$$
When $w= (0,0,0,0)$ we see that $n_0=4, n_1=0$, and so $S_w=\{1\}$ and $S_w^c=\{0\}$ with $s:=n - |S_w^c|=3$.
Thus we get $A_w: \frac{1}{2}, \frac{1}{4}, \frac{3}{4}$ and $B_w: 1,1,1$. 
Now when $w=(0,0,1,1)$ we get that $n_0=2, n_1=2$, and so $S_w=\emptyset$ and $S_w^c=\{0, 1\}$ with $s=2$.
Thus we get $A_w: \frac{1}{4}, \frac{3}{4}$ and $B_w: 1,\frac{1}{2}$.
So when $p \equiv 3 \pmod 4$ we get
\begin{equation*}
N_p(\lambda) = \frac{p^{3}-1}{p-1} + 
{_{3}G_{3}}
\biggl[ \begin{array}{ccc} \frac{1}{4} & \frac{1}{2} & \frac{3}{4} \\
 1 & 1 & 1 \end{array}
\Big| \; \lambda^4 \; \biggr]_p
-3p \cdot
{_{2}G_{2}}
\biggl[ \begin{array}{cc} \frac{1}{4} & \frac{3}{4} \\
 1 & \frac{1}{2} \end{array}
\Big| \; \lambda^4 \; \biggr]_p.
\end{equation*}
This corresponds to Theorem 1.2 in \cite{Go}.
%\underline{Case $p \equiv 1 \pmod{4}$:} 

If  $p \equiv 1 \pmod{4}$ then $d=4$. Here we have
$${W/\sim} = \{[0,0,0,0], [0,0,2,2]^3, [0,0,1,3]^{12} \}.$$
When $w= (0,0,0,0)$ we have $A_w: \frac{1}{2}, \frac{1}{4}, \frac{3}{4}$ and $B_w: 1,1,1$.
When $w= (0,0,2,2)$ we have $A_w: \frac{1}{4}, \frac{3}{4}$ and $B_w: 1,\frac{1}{2}$.
Finally, when $w= [0,0,1,3]$ we get $A_w: \frac{1}{2}$ and $B_w: 1$.
So when $p \equiv 1 \pmod 4$ we get
\begin{equation*}
N_p(\lambda) = \frac{p^{3}-1}{p-1} + 
{_{3}G_{3}}
\biggl[ \begin{array}{ccc} \frac{1}{4} & \frac{1}{2} & \frac{3}{4} \\
 1 & 1 & 1 \end{array}
\Big| \; \lambda^4 \; \biggr]_p
-3p \cdot
{_{2}G_{2}}
\biggl[ \begin{array}{cc} \frac{1}{4} & \frac{3}{4} \\
 1 & \frac{1}{2} \end{array}
\Big| \; \lambda^4 \; \biggr]_p
-12p \cdot
{_{1}G_{1}}
\biggl[ \begin{array}{c} \frac{1}{2}  \\
 1 \end{array}
\Big| \; \lambda^4 \; \biggr]_p.
\end{equation*}
This corresponds to Theorem 1.3 in \cite{Go}, after simplification of the final term.

%%%%%%%%%%%%%%%%%%%%%%%%%%%%%%%%%%%%%%%%%%%%%%%%%%
%%%%%%%%%%%%%%%%%%%%%%%%%%%%%%%%%%%%%%%%%%%%%%%%%%
%%%%%%%%%%%%%%%%%%%%%%%%%%%%%%%%%%%%%%%%%%%%%%%%%%
%%%%%%%%%%%%%%%%%%%%%%%%%%%%%%%%%%%%%%%%%%%%%%%%%%
%%%%%%%%%%%%%%%%%%%%%%%%%%%%%%%%%%%%%%%%%%%%%%%%%%
%%%%%%%%%%%%%%%%%%%%%%%%%%%%%%%%%%%%%%%%%%%%%%%%%%

\section{Preliminaries}\label{sec_Prelim}
Let $\mathbb{Z}_p$ denote the ring of $p$-adic integers, $\mathbb{Q}_p$ the field of $p$-adic numbers, $\bar{\mathbb{Q}_p}$ the algebraic closure of $\mathbb{Q}_p$, and $\mathbb{C}_p$ the completion of $\bar{\mathbb{Q}_p}$.

Let $\zeta_p$ be a fixed primitive $p$-th root of unity in $\bar{\mathbb{Q}_p}$. We define the additive character $\theta : \mathbb{F}_p \rightarrow \mathbb{Q}_p(\zeta_p)$ by $\theta(x):=\zeta_p^{x}$. 
%We note that $\mathbb{Q}_p$ contains all $({p-1})$-st roots of unity and in fact they are all in $\mathbb{Z}^{*}_p$. 
We note that $\mathbb{Z}^{*}_p$ contains all $({p-1})$-st roots of unity. 
Thus we can consider multiplicative characters of $\mathbb{F}_p^{*}$ to be maps $\chi: \mathbb{F}_p^{*} \to \mathbb{Z}_{p}^{*}$. 
Recall that for $\chi \in \widehat{\mathbb{F}_p^{*}}$, the Gauss sum $g(\chi)$ is defined by 
%\begin{equation*}
$g(\chi):= \sum_{x \in \mathbb{F}_p} \chi(x) \theta(x).$
%\end{equation*}

The following useful result gives a simple expression for the product of two Gauss sums. For $\chi \in \widehat{\mathbb{F}_p^{*}}$ we have
\begin{equation}\label{for_GaussConj}
g(\chi)g(\bar{\chi})=
\begin{cases}
\chi(-1) p & \text{if } \chi \neq \varepsilon,\\
1 & \text{if } \chi= \varepsilon.
\end{cases}
\end{equation}

In \cite{Ko4}, Koblitz provides a formula for the number of points in $\mathbb{P}^{n-1}(\mathbb{F}_p)$ on the hypersurface 
$x_1^a + x_2^a + \dots + x_n^a - a \lambda x_1 x_2 \dots x_n=0$,
for some $a, \lambda \in \mathbb{F}_p$, where $p \equiv 1 \pmod a$.
Koblitz's result builds on the work of Weil \cite{W} which provides a formula in the $\lambda=0$ case. Weil's results hold for all primes $p$ and it is a relatively straightforward exercise to extend Koblitz's result to all primes in the case $a=n$, as follows.
%Specifically, let $d:=\gcd(p-1,n)$ and 
%\begin{equation}\label{Def_W}
%W:=\{w=(w_1, w_2, \ldots, w_n) \in \mathbb{Z}^n : 0 \leq w_i < d, \sum_{i=1}^n w_i \equiv 0 \pmod d\}. 
%\end{equation}
%\noindent Let $T$ be a fixed generator for the group of characters of $\mathbb{F}_p^*$ and set $t:=\frac{p-1}{d}$. 
%Define an equivalence relation $\sim$ on $W$ by 
%\begin{equation}\label{Def_tilde}
%w \sim w \prime \textup{ if } w- w\prime \textup{ is a multiple modulo $d$ of $(1,1, \ldots, 1)$}. 
%\end{equation}
%Let $d:=\gcd(p-1,n)$, $t:=\frac{p-1}{d}$ and define 
%\begin{equation}\label{Def_Np0}
%N_p(0,w):=
%\begin{cases}
%0 & \text{if some but not all } w_i=0,\\[6pt]
%\frac{p^{n-1}-1}{p-1} & \text{if all } w_i=0,\\[6pt]
%\frac{1}{p} \prod_{i=1}^{n} g(T^{w_i t}) & \text{if all } w_i \neq 0.
%\end{cases}
%\end{equation}

%\noindent Then we get the following theorem.
\begin{theorem}[cf Koblitz \cite{Ko4} {Thm.~2}, Weil \cite{W}]\label{thm_Koblitz}
Let $N_p(\lambda)$ be the number of points in $\mathbb{P}^{n-1}(\mathbb{F}_p)$ on $\sum_{i=1}^n x_i^n - n \lambda \prod_{j=1}^{n} x_i=0$, for some $n \in \mathbb{F}_p^{*}$, $\lambda \in \mathbb{F}_p$. Let $T$ be a fixed generator for $\widehat{\mathbb{F}_p^{*}}$, $d:=\gcd(p-1,n)$ and $t:=\frac{p-1}{d}$. 
%Let $W$, $\sim$ and $N_p(0,w)$ be defined by (\ref{Def_W}), (\ref{Def_tilde}) and (\ref{Def_Np0}) respectively. 
Let $W$ be defined by (\ref{Def_W}). 
Then
\begin{equation*}
N_p(\lambda) = \sum_{w \in W} N_p(0,w) + \frac{1}{p-1} \sum_{w \in W} \; \sum_{j=0}^{t-1} 
\frac{\prod_{i=1}^{n} g(T^{w_i t + j})}{g(T^{nj})}\; T^{nj}(n \lambda).\\
\end{equation*}
where 
\begin{equation*}\label{Def_Np0}
N_p(0,w):=
\begin{cases}
0 & \text{if some but not all } w_i=0,\\[6pt]
\frac{p^{n-1}-1}{p-1} & \text{if all } w_i=0,\\[6pt]
\frac{1}{p} \prod_{i=1}^{n} g(T^{w_i t}) & \text{if all } w_i \neq 0.
\end{cases}
\end{equation*}
\end{theorem}
Theorem \ref{thm_Koblitz} can also be proved directly using the point counting technique in \cite{W}. This technique is also often used to establish results involving finite field hypergeometric functions \cite{BRS, McC8, F2, Go, L2, McC5}.

%\subsection{$p$-adic Preliminaries}\label{subsec_padicPrelim}
We define the Teichm\"{u}ller character to be the primitive character $\omega: \mathbb{F}_p \rightarrow\mathbb{Z}^{*}_p$ satisfying $\omega(x) \equiv x \pmod p$ for all $x \in \{0,1, \ldots, p-1\}$.
We now recall the $p$-adic gamma function. For further details, see \cite{Ko}.
Let $p$ be an odd prime.  For $n \in \mathbb{Z}^{+}$ we define the $p$-adic gamma function as
\begin{align*}
\gfp{n} &:= {(-1)}^n \prod_{\substack{0<j<n\\p \nmid j}} j \\
\intertext{and extend it to all $x \in\mathbb{Z}_p$ by setting $\gfp{0}:=1$ and} 
\gfp{x} &:= \lim_{n \rightarrow x} \gfp{n}
\end{align*}
for $x\neq 0$, where $n$ runs through any sequence of positive integers $p$-adically approaching $x$. 
This limit exists, is independent of how $n$ approaches $x$, and determines a continuous function
on $\mathbb{Z}_p$ with values in $\mathbb{Z}^{*}_p$.
We now state a product formula for the $p$-adic gamma function.
If $m\in\mathbb{Z}^{+}$, $p \nmid m$ and $x=\frac{r}{p-1}$ with $0\leq r \leq p-1$, then
\begin{equation}\label{for_pGammaMult}
\prod_{h=0}^{m-1} \gfp{\tfrac{x+h}{m}}=\omega\left(m^{(1-x)(1-p)}\right)
\gfp{x} \prod_{h=1}^{m-1} \gfp{\tfrac{h}{m}}.
\end{equation}
We note also that
\begin{equation}\label{for_pGammaOneMinus}
\gfp{x}\gfp{1-x} = {(-1)}^{x_0},
\end{equation}
where $x_0 \in \{1,2, \dotsc, {p}\}$ satisfies $x_0 \equiv x \pmod {p}$. The Gross-Koblitz formula \cite{GK} allows us to relate Gauss sums and the $p$-adic gamma function. Let $\pi \in \mathbb{C}_p$ be the fixed root of $x^{p-1}+p=0$ that satisfies ${\pi \equiv \zeta_p-1 \pmod{{(\zeta_p-1)}^2}}$. Then we have the following result.
\begin{theorem}[Gross, Koblitz \cite{GK}]\label{thm_GrossKoblitz}
For $ j \in \mathbb{Z}$,
\begin{equation*} 
g(\bar{\omega}^j)=-\pi^{(p-1) \langle{\frac{j}{p-1}}\rangle} \: \gfp{\langle{\tfrac{j}{p-1}}\rangle}.
\end{equation*}
\end{theorem}
\noindent We recall also the following result which can be derived from (\ref{for_pGammaMult}).
\begin{lemma}[\cite{McC7} Lemma 4.1] \label{lem_pGamma}
Let $p$ be prime. For $0 \leq j \leq p-2$ and $n \in \mathbb{Z}^{+}$ with $p \nmid n$,
%\begin{equation}\label{lem_gammaptj}
%\gfp{\Big\langle {\tfrac{tj}{p-1}} \Big\rangle}\; {\omega(t^{tj}) \displaystyle\prod_{h=1}^{t-1} \gfp{\tfrac{h}{t}}} = \displaystyle\prod_{h=0}^{t-1} \gfp{\Big\langle \tfrac{h}{t} + \tfrac{j}{p-1} \Big\rangle}
%\end{equation}
%and
\begin{equation*}\label{lem_gammamtj}
\gfp{\Big\langle {\tfrac{-nj}{p-1}} \Big\rangle}\; {\omega(n^{-nj}) \displaystyle\prod_{h=1}^{n-1} \gfp{\tfrac{h}{n}}} = \displaystyle\prod_{h=0}^{n-1} \gfp{\Big\langle \tfrac{1+h}{n} - \tfrac{j}{p-1} \Big\rangle}.
\end{equation*}
\end{lemma}

%%%%%%%%%%%%%%%%%%%%%%%%%%%%%%%%%%%%%%%%%%%%%%%%%%
%%%%%%%%%%%%%%%%%%%%%%%%%%%%%%%%%%%%%%%%%%%%%%%%%%
%%%%%%%%%%%%%%%%%%%%%%%%%%%%%%%%%%%%%%%%%%%%%%%%%%
%%%%%%%%%%%%%%%%%%%%%%%%%%%%%%%%%%%%%%%%%%%%%%%%%%
%%%%%%%%%%%%%%%%%%%%%%%%%%%%%%%%%%%%%%%%%%%%%%%%%%
%%%%%%%%%%%%%%%%%%%%%%%%%%%%%%%%%%%%%%%%%%%%%%%%%%

\section{Proofs}\label{sec_Proofs}

\begin{proof}[Proof of Theorem \ref{thm_Main}]
By Theorem \ref{thm_Koblitz} we have
\begin{equation*}
N_p(\lambda) = \frac{p^{n-1}-1}{p-1} + \frac{1}{p} \sum_{\substack{w \in W \\ w_i \neq 0}}  \prod_{i=1}^{n} g(T^{w_i t})+ \frac{1}{p-1} \sum_{w \in W} \; \sum_{j=0}^{t-1} 
\frac{\prod_{i=1}^{n} g(T^{w_i t + j})}{g(T^{nj})}\; T^{nj}(n \lambda).\\
\end{equation*}
From (\ref{for_GaussConj}) we get that
\begin{equation*}
g(T^{nj})g(T^{-nj})=
\begin{cases}
T^{nj}(-1) \; p & \text{if } T^{nj} \neq \varepsilon,\\
1 & \text{if } T^{nj}= \varepsilon.
\end{cases}
\end{equation*}
Now $T^{nj}= \varepsilon$ if and only if $j=0$, as $0 \leq j < \frac{p-1}{d}$. Therefore
\begin{align*}
\notag N_p(\lambda) 
&= \frac{p^{n-1}-1}{p-1} + \frac{1}{p} \sum_{\substack{w \in W \\ \text{all } w_i \neq 0}}  \prod_{i=1}^{n} g(T^{w_i t}) 
-\frac{1}{p-1} \sum_{w \in W}  \prod_{i=1}^{n} g(T^{w_i t}) \\
\notag & \qquad \qquad  \qquad \qquad +\frac{1}{p(p-1)} \sum_{w \in W} \; \sum_{j=1}^{t-1} 
\prod_{i=1}^{n} g(T^{w_i t + j}) \; g(T^{-nj})\; T^{nj}(-n \lambda)\\
%\notag &= \frac{p^{n-1}-1}{p-1} + \frac{1}{p} \sum_{\substack{w \in W \\ \text{all } w_i \neq 0}}  \prod_{i=1}^{n} g(T^{w_i t}) 
%-\frac{1}{p-1} \sum_{w \in W}  \prod_{i=1}^{n} g(T^{w_i t})  \\
%&  +\frac{1}{p(p-1)}  \sum_{w \in W}  \prod_{i=1}^{n} g(T^{w_i t}) +\frac{1}{p(p-1)} \sum_{w \in W} \; \sum_{j=0}^{t-1} 
%\prod_{i=1}^{n} g(T^{w_i t + j}) \; g(T^{-nj})\; T^{nj}(-n \lambda)\\
\notag &= \frac{p^{n-1}-1}{p-1} + \frac{1}{p} \sum_{\substack{w \in W \\ \text{all } w_i \neq 0}}  \prod_{i=1}^{n} g(T^{w_i t}) 
-\frac{1}{p-1} \sum_{w \in W}  \prod_{i=1}^{n} g(T^{w_i t}) \left[ 1 - \frac{1}{p} \right]\\
\notag & \qquad \qquad  \qquad \qquad +\frac{1}{p(p-1)} \sum_{w \in W} \; \sum_{j=0}^{t-1} 
\prod_{i=1}^{n} g(T^{w_i t + j}) \; g(T^{-nj})\; T^{nj}(-n \lambda)\\
\notag &= \frac{p^{n-1}-1}{p-1} - \frac{1}{p} \sum_{\substack{w \in W \\ \text{some } w_i=0}}  \prod_{i=1}^{n} g(T^{w_i t}) \\
\notag & \qquad \qquad  \qquad \qquad +\frac{1}{p(p-1)} \sum_{w \in W} \; \sum_{j=0}^{t-1} 
\prod_{i=1}^{n} g(T^{w_i t + j}) \; g(T^{-nj})\; T^{nj}(-n \lambda).\\
\notag &= \frac{p^{n-1}-1}{p-1} - \frac{1}{p} \sum_{\substack{w \in W \\ \text{some } w_i=0}}  \prod_{i=1}^{n} g(T^{w_i t}) \\
& \qquad \qquad  \qquad \qquad +\frac{1}{p(p-1)} \sum_{[w] \in W/\sim} \; \sum_{j=0}^{p-2} 
\prod_{i=1}^{n} g(T^{w_i t + j}) \; g(T^{-nj})\; T^{nj}(-n \lambda).
\end{align*}
We now examine the inner sum in the last term above, which we will denote $R_{[w]}$, i.e.,
\begin{equation*}
R_{[w]} = \sum_{j=0}^{p-2} 
\prod_{i=1}^{n} g(T^{w_i t + j}) \; g(T^{-nj})\; T^{nj}(-n \lambda),
\end{equation*}
and
\begin{equation}\label{Np_for1}
N_p(\lambda) = \frac{p^{n-1}-1}{p-1} - \frac{1}{p} \sum_{\substack{w \in W \\ \text{some } w_i=0}}  \prod_{i=1}^{n} g(T^{w_i t}) +\frac{1}{p(p-1)} \sum_{[w] \in W/\sim} \; R_{[w]}.
\end{equation}
We note that $R_{[w]}$ is independent of choice of representative for the equivalence class.
Recalling the notation from Section \ref{sec_Results} we see that
\begin{align*}
R_{[w]}  = \sum_{j=0}^{p-2} \; \prod_{k \in S_w^c} g(T^{k t + j})^{n_k} \; g(T^{-nj})\; T^{nj}(-n \lambda).
\end{align*}
Again using (\ref{for_GaussConj}) we get that
\begin{equation*}
g(T^{k t + j}) \, g(T^{-k t - j})=
\begin{cases}
T^{k t + j}(-1) \; p & \text{if } T^{k t + j} \neq \varepsilon,\\
1 & \text{if } T^{k t + j}= \varepsilon.
\end{cases}
\end{equation*}
Now $T^{k t + j}= \varepsilon$ if and only if $k=j=0$ or $k>0, j=(d-k)t$. So, as $n_k\geq1$ when $k \in S_w^c$,
\begin{align*}
R_{[w]}  &= \sum_{\substack{j=0\\j \not\equiv 0 \imod{t}}}^{p-2} \; 
\prod_{k \in S_w^c} \frac{g(T^{k t + j})^{n_k-1} \; T^{k t + j}(-1) \; p}{g(T^{-k t - j})}
 \; g(T^{-nj})\; T^{nj}(-n \lambda)\\
 &  \qquad \qquad \qquad \qquad -\sum_{a=0}^{d-1} \, \prod_{k \in S_w^c} \frac{g(T^{(k+a) t })^{n_k-1} }{g(T^{-(k+a)t})}
 \prod_{\substack{k \in S_w^c\\k \not\equiv -a \imod{d}}} \left(T^{(k+a) t}(-1) \; p \right)\\
&= \sum_{j=0}^{p-2} \; 
\prod_{k \in S_w^c} \frac{g(T^{k t + j})^{n_k-1} \; T^{k t + j}(-1) \; p}{g(T^{-k t - j})}
 \; g(T^{-nj})\; T^{nj}(-n \lambda)\\
  &  \qquad -\sum_{a=0}^{d-1} \, \prod_{k \in S_w^c} \frac{g(T^{(k+a) t })^{n_k-1} }{g(T^{-(k+a)t})}
 \left[ \prod_{\substack{k \in S_w^c\\k \not\equiv -a (d)}} \left(T^{(k+a) t}(-1) \; p \right)
 -  \prod_{k \in S_w^c} \left(T^{(k+a) t}(-1) \; p \right) \right].\\
\end{align*}
For a given $0 \leq a \leq d-1$ define 
$$v_a:=
\begin{cases}
0 & \text{if } a=0\\
d-a & \text{if } a>0.
\end{cases}$$
Then $0 \leq v_a \leq d-1$ and $v_a \equiv - a \pmod{d}$.
So
\begin{align*}
R_{[w]}
&= \sum_{j=0}^{p-2} \; 
\prod_{k \in S_w^c} \frac{g(T^{k t + j})^{n_k-1} \; T^{k t + j}(-1) \; p}{g(T^{-k t - j})}
 \; g(T^{-nj})\; T^{nj}(-n \lambda)\\
  &  \qquad + (p-1) \sum_{\substack{a=0\\v_a \in S_w^c}}^{d-1} \, \prod_{k \in S_w^c} \frac{g(T^{(k+a) t })^{n_k-1} }{g(T^{-(k+a)t})}
 \prod_{\substack{k \in S_w^c\\k \neq v_a}} \left(T^{(k+a) t}(-1) \; p \right).
\end{align*}
We will analyze the two terms appearing on the right-hand side of the above equation separately, and refer to them as $R_{[w]}^{\prime}$ and $R_{[w]}^{\prime\prime}$ respectively. It is easy to see that $R_{[w]}^{\prime}$ is independent of choice of equivalence class representative thus so is $R_{[w]}^{\prime\prime}$.
We note first that for a given $0 \leq a \leq d-1$, if $v_a \in S_w^c$ then
\begin{align*}
\prod_{k \in S_w^c} {g(T^{-(k+a)t})}  \prod_{\substack{k \in S_w^c\\k \neq v_a}} {g(T^{(k+a)t})}
&= - \prod_{\substack{k \in S_w^c\\k \neq v_a}} {g(T^{-(k+a)t})}  \prod_{\substack{k \in S_w^c\\k \neq v_a}} {g(T^{(k+a)t})}\\
&= - \prod_{\substack{k \in S_w^c\\k \neq v_a}} \left[ T^{(k+a)t}(-1) \, p \right]
\end{align*}
using (\ref{for_GaussConj}) and the fact that $g(\varepsilon)=-1$. Thus
\begin{align*}
R_{[w]}^{\prime\prime}
&=-(p-1) \sum_{\substack{a=0\\v_a \in S_w^c}}^{d-1} \, \prod_{k \in S_w^c} g(T^{(k+a) t })^{n_k-1}
 \prod_{\substack{k \in S_w^c\\k \neq v_a}} {g(T^{(k+a)t})}\\
&=(p-1) \sum_{\substack{a=0\\v_a \in S_w^c}}^{d-1} \, \prod_{k \in S_w^c} g(T^{(k+a) t })^{n_k}\\
&=(p-1) \sum_{\substack{a=0\\v_a \in S_w^c}}^{d-1} \, \prod_{i=1}^{n} g(T^{(w_i+a) t }).
\end{align*}
For a given $0 \leq a \leq d-1$ let $\bar{a}$ be the $n$-tuple $(a,a, \cdots, a)$. Note then that 
$$v_a \in S_w^c \Longleftrightarrow 0 \in S_{w+\bar{a}}^{c}$$
where the addition $w+\bar{a}$ is considered modulo $d$ so $w+\bar{a} \in W$.
Therefore
\begin{align*}
\sum_{[w] \in W/\sim} R_{[w]}^{\prime\prime}
&= (p-1) \sum_{[w] \in W/\sim} \sum_{\substack{a=0\\0 \in S_{w+\bar{a}}^{c}}}^{d-1} \, \prod_{i=1}^{n} g(T^{(w_i+a) t })\\
&=  (p-1) \sum_{\substack{w \in W \\ \text{some } w_i=0}}  \prod_{i=1}^{n} g(T^{w_i t }).
\end{align*}
So now (\ref{Np_for1}) becomes
\begin{align}\label{Np_for2}
\notag N_p(\lambda) 
\notag &= \frac{p^{n-1}-1}{p-1} - \frac{1}{p} \sum_{\substack{w \in W \\ \text{some } w_i=0}}  \prod_{i=1}^{n} g(T^{w_i t}) 
+\frac{1}{p(p-1)} \sum_{[w] \in W/\sim} \; \left(R_{[w]}^{\prime} +R_{[w]}^{\prime\prime} \right)\\
&= \frac{p^{n-1}-1}{p-1} +\frac{1}{p(p-1)} \sum_{[w] \in W/\sim} \; R_{[w]}^{\prime},
\end{align}
where
\begin{equation*}
R_{[w]}^{\prime} = \sum_{j=0}^{p-2} \; 
\prod_{k \in S_w^c} \frac{g(T^{k t + j})^{n_k-1} \; T^{k t + j}(-1) \; p}{g(T^{-k t - j})}
 \; g(T^{-nj})\; T^{nj}(-n \lambda).
\end{equation*}
We now switch to the $p$-adic setting to analyze $R_{[w]}^{\prime}$. We let $T=\bar{\omega}$ and use the Gross-Koblitz formula, Theorem \ref{thm_GrossKoblitz}, to get
\begin{equation*}
R_{[w]}^{\prime} = \sum_{j=0}^{p-2} \; 
\prod_{k \in S_w^c} \frac{\biggfp{\langle \frac{k}{d} +\frac{j}{p-1} \rangle}^{n_k-1} \; \bar{\omega}^{k t + j}(-1) \; p}{\biggfp{\langle -\frac{k}{d} -\frac{j}{p-1} \rangle}}
 \; \biggfp{\langle -\tfrac{nj}{p-1} \rangle} \; \bar{\omega}^{nj}(-n \lambda) \cdot \pi^{(p-1)x} \cdot (-1)^y,
\end{equation*}
where
\begin{align*}
x&=\sum_{k \in S_w^c} (n_k-1) \langle \tfrac{k}{d} +\tfrac{j}{p-1} \rangle - \sum_{k \in S_w^c}  \langle -\tfrac{k}{d} -\tfrac{j}{p-1} \rangle + \langle -\tfrac{nj}{p-1} \rangle,
\intertext{and}
y&=\sum_{k \in S_w^c} (n_k-1)- \sum_{k \in S_w^c}1 + 1.
\end{align*}
Using the facts that $\langle x \rangle=x- \left\lfloor x \right\rfloor$, 
\begin{equation*}
 \sum_{k \in S_w^c} n_k = \sum_{i=0}^{d-1} n_k = n,
\end{equation*}
and
\begin{equation*}
 \sum_{k \in S_w^c} k \, n_k = \sum_{i=0}^{n-1} w_i \equiv 0 \pmod{d},
\end{equation*}
it is easy to see that
\begin{align}\label{for_x}
x&=  \sum_{k \in S_w^c} \frac{k \, n_k}{d} - \sum_{k \in S_w^c} (n_k-1) \left\lfloor \tfrac{k}{d} +\tfrac{j}{p-1} \right\rfloor 
+ \sum_{k \in S_w^c} \left\lfloor  -\tfrac{k}{d} -\tfrac{j}{p-1} \right\rfloor - \left\lfloor -\tfrac{nj}{p-1} \right\rfloor \in \mathbb{Z},
\intertext{and}
\notag y&= n+1.
\end{align}
So
\begin{equation}\label{Rw1_for1}
R_{[w]}^{\prime} = \sum_{j=0}^{p-2} \; 
\prod_{k \in S_w^c} \frac{\biggfp{\langle \frac{k}{d} +\frac{j}{p-1} \rangle}^{n_k-1} \; \bar{\omega}^{k t + j}(-1) \; p}{\biggfp{\langle -\frac{k}{d} -\frac{j}{p-1} \rangle}}
 \; \biggfp{\langle -\tfrac{nj}{p-1} \rangle} \; \bar{\omega}^{nj}(-n \lambda) \cdot (-p)^{x} \cdot (-1)^{n+1}.
\end{equation}
From Lemma \ref{lem_gammamtj} we see that as $p\nmid n$,
\begin{equation}\label{lem_gammamtj_1} 
\gfp{\Big\langle {\tfrac{-nj}{p-1}} \Big\rangle} 
= 
\frac{\displaystyle\prod_{h=0}^{n-1} \gfp{\Big\langle \tfrac{1+h}{n} - \tfrac{j}{p-1} \Big\rangle}}
{{\omega(n^{-nj}) \displaystyle\prod_{h=1}^{n-1} \gfp{\tfrac{h}{n}}}}
=
\frac{\displaystyle\prod_{h=0}^{n-1} \gfp{\Big\langle \tfrac{h}{n} - \tfrac{j}{p-1} \Big\rangle}}
{{\omega(n^{-nj}) \displaystyle\prod_{h=1}^{n-1} \gfp{\tfrac{h}{n}}}}.
\end{equation}
Now
\begin{align}\label{lem_gammamtj_2} 
\notag \displaystyle\prod_{h=0}^{n-1} \gfp{\Big\langle \tfrac{h}{n} - \tfrac{j}{p-1} \Big\rangle}
&=
\displaystyle\prod_{k=0}^{d-1} \gfp{\Big\langle \tfrac{k}{d} - \tfrac{j}{p-1} \Big\rangle} 
\displaystyle\prod_{\substack{h=0\\h \not\equiv 0 \, (\frac{n}{d})}}^{n-1} \gfp{\Big\langle \tfrac{h}{n} - \tfrac{j}{p-1} \Big\rangle}\\
\notag &=
\displaystyle\prod_{k=1}^{d} \gfp{\Big\langle \tfrac{d-k}{d} - \tfrac{j}{p-1} \Big\rangle} 
\displaystyle\prod_{\substack{h=0\\h \not\equiv 0 \, (\frac{n}{d})}}^{n-1} \gfp{\Big\langle \tfrac{h}{n} - \tfrac{j}{p-1} \Big\rangle}\\
&=
\displaystyle\prod_{k=0}^{d-1} \gfp{\Big\langle -\tfrac{k}{d} - \tfrac{j}{p-1} \Big\rangle} 
\displaystyle\prod_{\substack{h=0\\h \not\equiv 0 \, (\frac{n}{d})}}^{n-1} \gfp{\Big\langle \tfrac{h}{n} - \tfrac{j}{p-1} \Big\rangle}
\end{align}
Accounting for (\ref{lem_gammamtj_1}) and (\ref{lem_gammamtj_2}) in (\ref{Rw1_for1}) we get
\begin{multline}\label{Rw1_for2}
R_{[w]}^{\prime} = \frac{(-1)^{n-1}}{\displaystyle\prod_{h=1}^{n-1} \gfp{\tfrac{h}{n}}} \;
 \sum_{j=0}^{p-2} \; 
\prod_{k \in S_w^c} \left(\biggfp{\langle \tfrac{k}{d} +\tfrac{j}{p-1} \rangle}^{n_k-1} \; \bar{\omega}^{k t + j}(-1) \; p \right)
\prod_{k \in S_w} \gfp{\langle -\tfrac{k}{d} - \tfrac{j}{p-1} \rangle} 
\\
\cdot \displaystyle\prod_{\substack{h=0\\h \not\equiv 0 \, (\frac{n}{d})}}^{n-1} \gfp{\langle \tfrac{h}{n} - \tfrac{j}{p-1} \rangle}
\cdot \bar{\omega}^{nj}(- \lambda) \cdot (-p)^{x} .
\end{multline}
Our aim now is to convert (\ref{Rw1_for2}) in to the appropriate $_mG_m[\cdots]$.
We start with a few preliminary observations. 
Similar to (\ref{lem_gammamtj_2}) we have
\begin{equation}\label{for_obs1}
\displaystyle\prod_{h=1}^{n-1} \gfp{\tfrac{h}{n}} 
= \displaystyle\prod_{k=1}^{d-1} \gfp{\tfrac{k}{d}} \displaystyle\prod_{\substack{h=0\\h \not\equiv 0 \, (\frac{n}{d})}}^{n-1} \gfp{\tfrac{h}{n}}.
\end{equation}
Furthermore, and noting that $\gfp{0}:=1$, we see that
\begin{equation}\label{for_obs2}
\displaystyle\prod_{k=1}^{d-1} \gfp{\tfrac{k}{d}} 
%= \displaystyle\prod_{k=1}^{d-1} \gfp{\tfrac{k}{d}} 
%= \displaystyle\prod_{k=1}^{d-1} \gfp{\tfrac{d-k}{d}}
= \displaystyle\prod_{k=1}^{d-1} \gfp{\langle \tfrac{d-k}{d} \rangle}
= \displaystyle\prod_{k=0}^{d-1} \gfp{\langle -\tfrac{k}{d} \rangle}
= \displaystyle\prod_{k \in S_w} \gfp{\langle -\tfrac{k}{d} \rangle}
\displaystyle\prod_{k \in S_w^c} \gfp{\langle -\tfrac{k}{d} \rangle}.
\end{equation}
Using the Gross-Koblitz formula (Theorem  \ref{thm_GrossKoblitz}) and (\ref{for_GaussConj}) we also observe that 
\begin{align}\label{for_obs3}
\notag \displaystyle\prod_{k \in S_w^c}  \gfp{\langle -\tfrac{k}{d} \rangle} \; \gfp{\langle \tfrac{k}{d} \rangle}
%&=
%\displaystyle\prod_{k \in S_w^c} \gfp{\langle -\tfrac{kt}{p-1} \rangle} \; \gfp{\langle \tfrac{kt}{p-1} \rangle}\\
&=
\displaystyle\prod_{k \in S_w^c} g(\bar{\omega}^{-k t}) g(\bar{\omega}^{k t}) 
\; \pi^{-(p-1) \left[ \langle\frac{-k}{d}\rangle + \langle\frac{k}{d}\rangle \right] }\\
\notag &=
(-p)^{-|S_w^c\setminus \{0\}|} \displaystyle\prod_{k \in S_w^c\setminus \{0\}} \left( \bar{\omega}^{k t}(-1) \, p \right) \\
&=
(-1)^{|S_w^c\setminus \{0\}|} \displaystyle\prod_{k \in S_w^c} \bar{\omega}^{k t}(-1),
\end{align}
as 
$$\langle\tfrac{-k}{d}\rangle + \langle\tfrac{k}{d}\rangle = - \lfloor \tfrac{-k}{d}\rfloor - \lfloor\tfrac{k}{d}\rfloor
=
-
\begin{cases}
0 & k=0\\
-1 & 1 \leq k \leq {d-1}.
\end{cases}
$$
Combining (\ref{for_obs1}), (\ref{for_obs2}) and (\ref{for_obs3}) we get that
\begin{equation*}
\frac
{
\displaystyle\prod_{k \in S_w^c} \biggfp{\langle \tfrac{k}{d} \rangle}^{n_k-1} 
\displaystyle\prod_{k \in S_w} \gfp{\langle -\tfrac{k}{d}  \rangle}
\displaystyle\prod_{\substack{h=0\\h \not\equiv 0 \, (\frac{n}{d})}}^{n-1} \gfp{\langle \tfrac{h}{n}  \rangle}
}
{\displaystyle\prod_{h=1}^{n-1} \gfp{\tfrac{h}{n}}}
=
\frac
{\displaystyle\prod_{k \in S_w^c} \biggfp{\langle \tfrac{k}{d} \rangle}^{n_k}}
{(-1)^{|S_w^c\setminus \{0\}|} \displaystyle\prod_{k \in S_w^c} \bar{\omega}^{k t}(-1)},
\end{equation*}
and so (\ref{Rw1_for2}) becomes
\begin{multline}\label{Rw1_for3}
R_{[w]}^{\prime} = (-1)^{n-1}
 \sum_{j=0}^{p-2} \; 
 (-1)^{|S_w^c\setminus \{0\}|}
\displaystyle\prod_{k \in S_w^c} \frac{\biggfp{\langle \tfrac{k}{d} +\tfrac{j}{p-1} \rangle}^{n_k-1}}{\biggfp{\langle \tfrac{k}{d} \rangle}^{n_k-1} }
\displaystyle\prod_{k \in S_w} \frac{\gfp{\langle -\tfrac{k}{d} - \tfrac{j}{p-1} \rangle}}{\gfp{\langle -\tfrac{k}{d}  \rangle}}
\\ \cdot
\displaystyle\prod_{\substack{h=0\\h \not\equiv 0 \, (\frac{n}{d})}}^{n-1} \frac{\gfp{\langle \tfrac{h}{n} - \tfrac{j}{p-1} \rangle}}
{\gfp{\langle \tfrac{h}{n}  \rangle}}
\displaystyle\prod_{k \in S_w^c} \biggfp{\langle \tfrac{k}{d} \rangle}^{n_k}
\cdot  \left(\bar{\omega}^{ j}(-1) \; p \right)^{|S_w^c|}
\cdot \bar{\omega}^{nj}(- \lambda) \cdot (-p)^{x} 
\end{multline}
We now turn our attention to the power of $(-p)$. Comparing (\ref{Rw1_for3}) to Definition \ref{def_Gp} for $_mG_m[\cdots]$ we see that for our particular arguments of the $p$-adic gamma function we would like the power of $(-p)$ to be
$$z := - \left[\sum_{k \in S_w^c} (n_k-1) \lfloor \langle \tfrac{k}{d} \rangle +\tfrac{j}{p-1} \rfloor
+\sum_{k \in S_w} \lfloor \langle -\tfrac{k}{d} \rangle - \tfrac{j}{p-1} \rfloor
+\sum_{\substack{h=0\\h \not\equiv 0 \, (\frac{n}{d})}}^{n-1} \lfloor \langle \tfrac{h}{n} \rangle - \tfrac{j}{p-1} \rfloor
\right]. $$
Comparing to (\ref{for_x}) we see that
\begin{multline*}
x-z= \sum_{k \in S_w^c} \frac{k \, n_k}{d} 
%- \sum_{k \in S_w^c} (n_k-1) \left\lfloor \tfrac{k}{d} +\tfrac{j}{p-1} \right\rfloor 
+ \sum_{k \in S_w^c} \left\lfloor  -\tfrac{k}{d} -\tfrac{j}{p-1} \right\rfloor 
- \left\lfloor -\tfrac{nj}{p-1} \right\rfloor
+\sum_{k \in S_w} \lfloor \langle -\tfrac{k}{d} \rangle - \tfrac{j}{p-1} \rfloor
+\sum_{\substack{h=0\\h \not\equiv 0 \, (\frac{n}{d})}}^{n-1} \lfloor \langle \tfrac{h}{n} \rangle - \tfrac{j}{p-1} \rfloor.
\end{multline*}
A straightforward calculation yields
\begin{equation*}
\left\lfloor -\tfrac{nj}{p-1} \right\rfloor = \sum_{h=0}^{n-1} \lfloor \tfrac{h}{n} - \tfrac{j}{p-1} \rfloor
= \sum_{\substack{h=0\\h \not\equiv 0 \, (\frac{n}{d})}}^{n-1} \lfloor \langle \tfrac{h}{n} \rangle - \tfrac{j}{p-1} \rfloor
+ \sum_{k=0}^{d-1} \lfloor \langle -\tfrac{k}{d} \rangle - \tfrac{j}{p-1} \rfloor.
\end{equation*}
So
\begin{align*}
x-z
&=\sum_{k \in S_w^c} \frac{k \, n_k}{d}
+ \sum_{k \in S_w^c} \left\lfloor  -\tfrac{k}{d} -\tfrac{j}{p-1} \right\rfloor 
-\sum_{k \in S_w^c} \lfloor \langle -\tfrac{k}{d} \rangle - \tfrac{j}{p-1} \rfloor\\
%&=\sum_{k \in S_w^c} \frac{k \, n_k}{d}
%+ \sum_{k \in S_w^c\setminus \{0\}} \left\lfloor  -\tfrac{k}{d} -\tfrac{j}{p-1} \right\rfloor 
%-\sum_{k \in S_w^c\setminus \{0\}} \lfloor \langle -\tfrac{k}{d} \rangle - \tfrac{j}{p-1} \rfloor\\
&=\sum_{k \in S_w^c} \frac{k \, n_k}{d}
+ \sum_{k \in S_w^c\setminus \{0\}} \left\lfloor  -\tfrac{k}{d} -\tfrac{j}{p-1} \right\rfloor 
-\sum_{k \in S_w^c\setminus \{0\}} \lfloor  \tfrac{d-k}{d} - \tfrac{j}{p-1} \rfloor\\
%&=\sum_{k \in S_w^c} \frac{k \, n_k}{d}
%+ \sum_{k \in S_w^c\setminus \{0\}} 1\\
&=\sum_{k \in S_w^c} \frac{k \, n_k}{d} - {|S_w^c\setminus \{0\}|}.
\end{align*}
Recall $s:= n-|S_w^c|$, $\bar{\omega}(-1)=-1$ and $\sum_{k \in S_w^c} k \, n_k = \sum_{i=1}^{n} w_i$. Therefore
\begin{multline*}
R_{[w]}^{\prime} = 
\displaystyle\prod_{i=1}^{n} \biggfp{\tfrac{w_i}{d}}
\cdot (-p)^{\frac{1}{d} \sum_{i=1}^{n} w_i} 
\cdot p^{\delta_w}
\cdot (-1)^{n-1} \\
 \times \sum_{j=0}^{p-2} \; 
 (-1)^{js} \; \bar{\omega}^j(\lambda^n)
\displaystyle\prod_{k \in S_w} \frac{\gfp{\langle \tfrac{d-k}{d} - \tfrac{j}{p-1} \rangle}}{\gfp{\langle \tfrac{d-k}{d}  \rangle}}
\displaystyle\prod_{\substack{h=0\\h \not\equiv 0 \, (\frac{n}{d})}}^{n-1} \frac{\gfp{\langle \tfrac{h}{n} - \tfrac{j}{p-1} \rangle}}
{\gfp{\langle \tfrac{h}{n}  \rangle}}
\\ \cdot
\displaystyle\prod_{k \in S_w^c} \frac{\biggfp{\langle -\tfrac{d-k}{d} +\tfrac{j}{p-1} \rangle}^{n_k-1}}{\biggfp{\langle -\tfrac{d-k}{d} \rangle}^{n_k-1} }
\cdot (-p)^z 
\end{multline*}
where
\begin{equation*}
\delta_w :=
\begin{cases}
1 & \text{if } 0 \in S_w^c\\
0 & \text{if } 0 \in S_w.
\end{cases}
\end{equation*}
So
\begin{equation*}
R_{[w]}^{\prime} = 
\displaystyle\prod_{i=1}^{n} \biggfp{\tfrac{w_i}{d}}
\cdot (-p)^{\frac{1}{d} \sum_{i=1}^{n} w_i} 
\cdot p^{\delta_w}
\cdot (-1)^{n}
\cdot (p-1) \cdot
{_{s}G_{s}}
\biggl[ \begin{array}{c} A_w \\
 B_w \end{array}
\Big| \; \lambda^n \; \biggr]_p,
\end{equation*}
where $A_w$ and $B_w$ are the parameter lists defined in (\ref{def_Aw}) and (\ref{def_Bw}). Therefore, from (\ref{Np_for2}) we get that
\begin{equation}
\notag N_p(\lambda) 
= \frac{p^{n-1}-1}{p-1} +
 \sum_{[w] \in W/\sim} \;
\displaystyle\prod_{i=1}^{n} \biggfp{\tfrac{w_i}{d}}
\cdot (-p)^{\frac{1}{d} \sum_{i=1}^{n} w_i} 
\cdot p^{\delta_w-1}
\cdot (-1)^{n} \cdot
{_{s}G_{s}}
\biggl[ \begin{array}{c} A_w \\
 B_w \end{array}
\Big| \; \lambda^n \; \biggr]_p.
\end{equation}
If the representative $w$ we chose in each equivalence class is such that $w_i=0$ for some $1 \leq i \leq n$ then $\delta_w=1$. Therefore, choosing only representatives of this form yields
\begin{equation}
\notag N_p(\lambda) 
= \frac{p^{n-1}-1}{p-1} +
 \sum_{[w^{\ast}] \in W/\sim} \;
\displaystyle\prod_{i=1}^{n} \biggfp{\tfrac{w_i}{d}}
\cdot (-p)^{\frac{1}{d} \sum_{i=1}^{n} w_i} 
\cdot (-1)^{n} \cdot
{_{s}G_{s}}
\biggl[ \begin{array}{c} A_w \\
 B_w \end{array}
\Big| \; \lambda^n \; \biggr]_p.
\end{equation}
\end{proof}

\begin{proof}[Proof of Corollary \ref{cor_RelPrime}]
If $d=1$ then $w=(0,0,\cdots, 0)$ is the only element  in $W$, and
$S_w= \emptyset$ and  $S_w^c=\{0\}.$
Thus
$A_w$ is  $\tfrac{1}{n}, \tfrac{2}{n}, \cdots , \tfrac{n-1}{n}$ and $B_w$ is $1, 1, \cdots, 1$ (${n-1}$ times).
Therefore in this case the result in Theorem \ref{thm_Main} reduces to
\begin{equation}
\notag N_p(\lambda) 
= \frac{p^{n-1}-1}{p-1} +
(-1)^{n} \cdot
{_{n-1}G_{n-1}}
\biggl[ \begin{array}{cccc} \frac{1}{n} & \frac{2}{n} & \dots & \frac{n-1}{n} \\
 1 & 1 & \dots & 1 \end{array}
\Big| \; \lambda^n \; \biggr]_p.
\end{equation}
\end{proof}

\begin{proof}[Proof of Corollary \ref{cor_FF}]
Corollary \ref{cor_FF} can be proved as a stand alone result, in a similar manner to Theorem \ref{thm_Main} but without having to transfer to the $p$-adic setting. But having proved Theorem \ref{thm_Main} above, we now derive Corollary \ref{cor_FF} from that result.

When $d=n$ then the list
$\left[ \tfrac{h}{n} \mid 0 \leq h \leq n-1, h \not\equiv 0 \imod{\tfrac{n}{d}} \right]$
is empty and so we get the lists
\begin{equation*}
A_w: \left[ \tfrac{n-k}{n} \mid k \in S_w \right] ; \textup{ and } B_w: \left[  \tfrac{n-k}{n} \, \textup{repeated $n_k$-1 times} \mid k \in S_w^c \right] .
\end{equation*}
%\begin{equation}
%B_w: \left[  \tfrac{n-k}{n} \, \textup{repeated $n_k$-1 times} \mid k \in S_w^c \right] .
%\end{equation}
We note that $|A_w|=|S_w| = d- |S_w^c| = n- |S_w^c| = \sum_{k \in S_w^c} (n_k - 1) = |B_w|$.
%We now note the relationship between $_mF_m(\cdots)$ and $_mG_m[\cdots]$ in the following lemma.
%\begin{lemma}[\cite{McC7} Lemma 3.3]\label{lem_G_to_F}
%For a fixed odd prime $p$, let $A_i, B_k \in \widehat{\mathbb{F}_p^{*}}$ be given by $\bar{\omega}^{a_i(p-1)}$ and $\bar{\omega}^{b_k(p-1)}$ respectively, where $\omega$ is the Teichm\"{u}ller character . Then
%\begin{equation*}
%{_{m}F_{m}} {\biggl( \begin{array}{cccc} A_1, & A_2, & \dotsc, & A_m \\
% B_1, & B_2, & \dotsc, & B_m \end{array}
%\Big| \; t \biggr)}_{p}
%=
%{_{m}G_{m}}
%\biggl[ \begin{array}{cccc} a_1, & a_2, & \dotsc, & a_m \\
% b_1, & b_2, & \dotsc, & b_m \end{array}
%\Big| \; t^{-1} \; \biggr]_p.
%\end{equation*}
%\end{lemma}
Then using Lemma \ref{lem_G_to_F} and Theorem \ref{thm_GrossKoblitz} we see that in this case Theorem \ref{thm_Main} reduces to
\begin{equation}\label{cor_FF_for1}
N_p(\lambda) = \frac{p^{n-1}-1}{p-1} +  \sum_{[w^{\ast}] \in W/\sim}  
\prod_{i=1}^{n} g(\bar{\omega}^{w_i t})\;
{_{|S_w|}F_{|S_w|}} {\biggl( \begin{array}{c} A_{\bar{\omega},w}^{\prime} \\
 B_{\bar{\omega},w}^{\prime} \end{array}
\Big| \; \lambda^{-n} \; \biggr)}_{p}.\\
\end{equation}
This equation holds if we replace $\bar{\omega}$ by any generator $T$ for $\widehat{\mathbb{F}^{*}_{p}}$. To see this, let $T = \bar{\omega}^{\alpha}$ for some $0 \leq \alpha \leq {p-2}$ with $\gcd(\alpha, {p-1})=1$. Define a map 
$f_{\alpha}: \; {W/\sim} \, \to \, {W/\sim}$
given by
$$ \quad f_{\alpha}[w] = [\alpha w \imod{n}],$$
%\begin{align*}
%f_{\alpha}:& \; {W/\sim} \; \to \; {W/\sim}\\
%\intertext{given by}
% \quad f_{\alpha}& [w] = [\alpha w \imod{n}],
%\end{align*}
where if $w=(w_1, w_2, \cdots w_n)$ then $\alpha w \imod{n} = (\alpha w_1 \imod{n}, \alpha w_2 \imod{n}, \cdots, \alpha w_n \imod{n})$. Then, as $\gcd(\alpha, n)=1$, $f_{\alpha}$ is a well-defined isomorphism on ${W/\sim}$. Now replacing $[w]$ by $[\alpha w \imod{n}]$ in (\ref{cor_FF_for1}) yields the result.
\end{proof}

%%%%%%%%%%%%%%%%%%%%%%%%%%%%%%%%%%%%%%%%%%%%%%%%%%
%%%%%%%%%%%%%%%%%%%%%%%%%%%%%%%%%%%%%%%%%%%%%%%%%%
%%%%%%%%%%%%%%%%%%%%%%%%%%%%%%%%%%%%%%%%%%%%%%%%%%
%%%%%%%%%%%%%%%%%%%%%%%%%%%%%%%%%%%%%%%%%%%%%%%%%%
%%%%%%%%%%%%%%%%%%%%%%%%%%%%%%%%%%%%%%%%%%%%%%%%%%
%%%%%%%%%%%%%%%%%%%%%%%%%%%%%%%%%%%%%%%%%%%%%%%%%%

%\section{Concluding Remarks}\label{sec_cr}

%%%%%%%%%%%%%%%%%%%%%%%%%%%%%%%%%%%%%%%%%%%%%%%%%%
%%%%%%%%%%%%%%%%%%%%%%%%%%%%%%%%%%%%%%%%%%%%%%%%%%
%%%%%%%%%%%%%%%%%%%%%%%%%%%%%%%%%%%%%%%%%%%%%%%%%%
%%%%%%%%%%%%%%%%%%%%%%%%%%%%%%%%%%%%%%%%%%%%%%%%%%

\end{document}